\documentclass[12pt,english]{article}

\usepackage{amsmath}
\usepackage{amsthm}
\usepackage{amssymb}
\usepackage{newpxmath, newpxtext}
\usepackage{enumerate}
\usepackage{graphicx}

\usepackage{pstricks-add}
\usepackage{pst-poly}

\allowdisplaybreaks

\newcommand{\Z}{\mathbb{Z}}
\newcommand{\R}{\mathbb{R}}
\newcommand{\F}{\mathbb{F}_3}
\theoremstyle{definition}
\newtheorem{thm}{Theorem}
\newtheorem{prop}{Proposition}
\newtheorem{cor}{Corollary}
\newtheorem*{proof*}{Proof}
\newtheorem{Ex}{Example}

\begin{document}
\begin{center}
{\Large A characteristic polynomial of the Seidel matrix over $\F$}
\vskip 5mm
\medskip
YUYA SUGISHITA \footnote{Department of Pure and Applied Mathematics, Graduate School of Information Science and Technology, Osaka University; \tt{u977762e@ecs.osaka-u.ac.jp}}
\vspace{0.5cm}\\
\vskip 5mm
\end{center}

\begin{abstract}
In this paper, we consider the condition when the characteristic polynomials of the Seidel matrix of the graphs are decomposed into products of linear polynomials over $\F$. We also show the equality over $\F$ between the characteristic polynomial of the Seidel matrix of a graph and the characteristic polynomial of the adjacency matrix of that graph.
\vskip 5mm
\end{abstract}

\begin{flushleft}
\vspace{0.2cm}
\end{flushleft}

\section{Introduction}
In this paper, we consider undirected simple finite graphs, and denote graphs by $X$ and $Y$, etc. The Seidel matrix and its eigenvalues (see, e.g., \cite{7} and \cite{6}) are useful for studying equiangular lines (see, e.g.,  \cite{3} and \cite{4}) and strongly regular graphs (see, e.g.,  \cite{1}). 

It is stated in \cite{4} that there is no Seidel matrix with the following characteristic polynomials:
\begin{align*}
&(x+5)^{33}(x-9)^{10}(x-11)^5(x^2-20x+95),\\
&(x+5)^{33}(x-9)^{12}(x-11)^4(x-13),\\
&(x+5)^{33}(x-7)(x-9)^9(x-11)^7.
\end{align*}
This indicates that there cannot be $50$ equiangular lines in $\R^{17}$. It is stated in \cite{7} that the spectrum of the Seidel matrices of a $k$-regular graph with $n$ vertices is completely determined from the spectrum of the adjacency matrix of that graph by the following equation: 
\begin{align*}
\det(xI-S(X))=(-1)^n2^n\dfrac{x+1+2k-n}{x+1+2k}\det\left(-\dfrac{x+1}{2}I-A(X)\right).
\end{align*}
We call the eigenvalue of the adjacency matrix \textit{main} if its eigenspace is not orthogonal to the all-ones vector. 
In \cite{6}, it is stated that the main eigenvalue of $S(X)$ and its eigenvectors can be restored from the eigenvalues of $A(X)$, and that the main eigenvalue of $A(X)$ and its eigenvectors can be restored from the eigenvalues of $S(X)$. Also, as in \cite{5}, the Seidel matrix itself is the subject of research. Let $\{\lambda_i\}_{i=1}^n$ be the eigenvalues of the Seidel matrix $S(X)$ and define the \textit{Seidel energy} of $X$ as $\sum_{i=1}^{n}\left|\lambda_i\right|$, moreover $\sum_{i=1}^{n}\left|\lambda_i\right|^\alpha$ is a generalization of Seidel energy of $X$. It is stated in \cite{5} the following inequality is obtained: $\sum_{i=1}^{n}\left|\lambda_i\right|^\alpha\leq (n-1)^\alpha+(n-1)$ for all $\alpha\geq 2$, where $X$ be a graph with $n\geq 3$ vertices, and $\left|\det S(X)\right|\geq n-1$.

The following problems can be considered naturally: 
\begin{enumerate}[(A)]
\item Find the relational expression over $\F$ between the characteristic polynomial of the Seidel matrix and the characteristic polynomial of the adjacency matrix.
\item Determine the exponent of each factor when the characteristic polynomial of the Seidel matrix over $\F$ can be decomposed into the product of linear polynomials.
\item Characterize graphs such that the characteristic polynomial of the Seidel matrix of disjoint union of graphs  $X$ and $Y$ can be represented by the product of the characteristic polynomials of the Seidel matrices of $X$ and $Y$.
\end{enumerate}

In this paper, we partially solve some of these problems. The following are the main theorems. Theorem \ref{th1} partially solves (A) of the above problems, and Theorem \ref{th2} partially solves (B). Also, in Proposition \ref{A} of Section 2, we partially solve (C).

We denote the characteristic polynomial of the matrix $M$ over $\F$ by $\varphi(M)$. We use $\varphi(M,x)$ instead of $\varphi(M)$ when we  want to specify the variable $x$. We denote the adjacency matrix of graph $X$ by $A(X)$, which is defined as $A(X)_{ij}=1$ if the vertices $i$ and $j$ are adjacent, and $A(X)_{ij}=0$ otherwise. Denote the Seidel matrix $J-I-2A(X)$ by $S(X)$, where $J$ is the all-ones matrix and $I$ is the identity matrix. A complete graph on $n$ vertices is denoted by $K_n$. The disjoint union of the graphs $X=(V(X), E(X))$ and $Y=(V(Y), E(Y))$ is denoted by $X\sqcup Y:=(V(X)\sqcup V(Y), E(X)\sqcup E(Y))$. Here, $V(X)$ and $V(Y)$ are sets of vertices of $X$ and $Y$, respectively, and $E(X)$ and $E(Y)$ are sets of edges of $X$ and $Y$, respectively. Also, we denote $n$ disjoint unions of graph $X$ by $nX$.

\begin{thm}\label{th1}
The following holds.
\begin{enumerate}
\item $\varphi(S(3X),x)=\varphi(A(X),x+1)^3$.
\item $\varphi(S(\overline{3X}),x)=\varphi(A(X),-x+1)^3$.
\end{enumerate}
\end{thm}
\begin{thm}\label{th2}
Let $r,s,t\in\Z_{\geq 0}$.
\begin{enumerate}
\item If $x^r(x-1)^s(x+1)^t$ is a characteristic polynomial of the Seidel matrix over $\F$ of a graph, then
\begin{itemize}
\item $r\equiv s\equiv t\equiv 0\mod 3$, or
\item $r\equiv 0, s\equiv t\equiv 1\mod 3$, or
\item $r\equiv 1, s\equiv t\equiv 0\mod 3$.
\end{itemize}
\item If $r\leq s+t$ and
\begin{itemize}
\item $r\equiv s\equiv t\equiv 0\mod 3$, or
\item $r\equiv 0, s\equiv t\equiv 1\mod 3$, or
\item $r\equiv 1, s\equiv t\equiv 0\mod 3$,
\end{itemize}
then $x^r(x-1)^s(x+1)^t$ is a characteristic polynomial of the Seidel matrix over $\F$ of a graph.
\end{enumerate}
\end{thm}

In Section 2, we define some notations used in this paper and then prepare to prove these two main theorems by computing the characteristic polynomials of various Seidel matrices. In Section 3, we state the two main theorems. In Section 4, we examine cases not covered by the main theorem. 

%And for problem C, we take a step forward.

\section{Preparation}
In this section, we compute the characteristic polynomial  of the Seidel matrix over $\F$ for some graphs.

\begin{Ex}
Here are some examples of the Seidel matrix.
\begin{align*}
S(K_3)&=\begin{pmatrix}
\begin{array}{cccccc} 
0&-1&-1\\
-1&0&-1\\
-1&-1&0\\
\end{array}
\end{pmatrix},
S(3K_2)=\begin{pmatrix}
\begin{array}{cccccc} 
0&-1&1&1&1&1\\
-1&0&1&1&1&1\\
1&1&0&-1&1&1\\
1&1&-1&0&1&1\\
1&1&1&1&0&-1\\
1&1&1&1&-1&0\\
\end{array}
\end{pmatrix},\\
S(\overline{K_3})&=\begin{pmatrix}
\begin{array}{cccccc} 
0&1&1\\
1&0&1\\
1&1&0\\
\end{array}
\end{pmatrix},
S(\overline{3K_2})=\begin{pmatrix}
\begin{array}{cccccc} 
0&1&-1&-1&-1&-1\\
1&0&-1&-1&-1&-1\\
-1&-1&0&1&-1&-1\\
-1&-1&1&0&-1&-1\\
-1&-1&-1&-1&0&1\\
-1&-1&-1&-1&1&0\\
\end{array}
\end{pmatrix}.
\end{align*}
Also, calculate the characteristic polynomial of the above matrices:
\begin{align*}
\varphi(S(K_3))&=(x-1)^3, 
\varphi(S(3K_2))=x^3(x-1)^3,\\
\varphi(S(\overline{K_3}))&=(x+1)^3, \varphi(S(\overline{3K_2}))=x^3(x+1)^3.
\end{align*}
\end{Ex}

The adjacency matrix of the disjoint union of two graphs is the direct sum of each adjacency matrix. On the other hand, that is not true for the Seidel matrix. Hence, the characteristic polynomial of the Seidel matrix of that graph is non-trivial.

\begin{prop}\label{A}
For any graphs $X$ and $Y$, the following holds.
\begin{enumerate}
\item $\varphi(S(3X\sqcup Y))=\varphi(S(3X))\varphi(S(Y))$.
\item $\varphi(S(\overline{3X}\sqcup Y))=\varphi(S(\overline{3X}))\varphi(S(Y))$.
\end{enumerate}
\end{prop}
\begin{proof}~
\begin{enumerate}
\item Note that $2J=-J$ because we are thinking in over $\F$.
\begin{align*}
\varphi(S(3X\sqcup Y))&=
\begin{vmatrix}
\begin{array}{cccc} 
xI-S(X)&-J&-J&-J\\
-J&xI-S(X)&-J&-J\\
-J&-J&xI-S(X)&-J\\
-J&-J&-J&xI-S(Y)\\
\end{array}
\end{vmatrix}\\
&=\begin{vmatrix}
\begin{array}{cccc} 
xI-S(X)+J&O&-J&-J\\
-(xI-S(X))-J&xI-S(X)+J&-J&-J\\
O&-(xI-S(X))-J&xI-S(X)&-J\\
O&O&-J&xI-S(Y)\\
\end{array}
\end{vmatrix}\\
&=\begin{vmatrix}
\begin{array}{cccc} 
xI-S(X)+J&O&-J&-J\\
O&xI-S(X)+J&J&J\\
O&O&xI-S(X)+J&O\\
O&O&-J&xI-S(Y)\\
\end{array}
\end{vmatrix}\\
&=\begin{vmatrix}
\begin{array}{c} 
xI-S(X)+J
\end{array}
\end{vmatrix}^3
\begin{vmatrix}
\begin{array}{c} 
xI-S(Y)
\end{array}
\end{vmatrix}.
\end{align*}
On the other hand, 
\begin{align}\label{E1}
\varphi(S(3X))\nonumber&=
\begin{vmatrix}
\begin{array}{ccc} 
xI-S(X)&-J&-J\\
-J&xI-S(X)&-J\\
-J&-J&xI-S(X)\\
\end{array}
\end{vmatrix}\\
\nonumber&=\begin{vmatrix}
\begin{array}{ccc} 
xI-S(X)+J&-J&-J\\
-(xI-S(X))-J&xI-S(X)&-J\\
O&-J&xI-S(X)\\
\end{array}
\end{vmatrix}\\
\nonumber&=\begin{vmatrix}
\begin{array}{ccc} 
xI-S(X)+J&-J&-J\\
O&xI-S(X)-J&J\\
O&-J&xI-S(X)\\
\end{array}
\end{vmatrix}\\
\nonumber&=\begin{vmatrix}
\begin{array}{ccc} 
xI-S(X)+J&O&-J\\
O&xI-S(X)+J&J\\
O&-(xI-S(X))-J&xI-S(X)\\
\end{array}
\end{vmatrix}\\
&=\begin{vmatrix}
\begin{array}{c} 
xI-S(X)+J
\end{array}
\end{vmatrix}^3.
\end{align}
Therefore, the desired formula is obtained.
\item Note that $S(\overline{X})=-S(X)$. Similar to the above calculation, 
\begin{align*}\varphi(S(\overline{3X}\sqcup Y))=\begin{vmatrix}
\begin{array}{c} 
xI+S(X)-J
\end{array}
\end{vmatrix}^3
\begin{vmatrix}
\begin{array}{c} 
xI-S(Y)
\end{array}
\end{vmatrix}.
\end{align*}
Also,  
\begin{align}\label{E2}
\varphi(S(\overline{3X}))
=\varphi(-S(3X))
=\begin{vmatrix}
\begin{array}{c} 
xI+S(X)-J
\end{array}
\end{vmatrix}^3.
\end{align}
Therefore, the desired formula is obtained.
\end{enumerate}
\end{proof}

\begin{cor}\label{B}
The following holds.
\begin{enumerate}
\item $\varphi(S(\overline{K_3}\sqcup Y))=\varphi(S(\overline{K_3}))\varphi(S(Y))=(x+1)^3\varphi(S(Y))$.
\item $\varphi(S(3K_2\sqcup Y))=\varphi(S(3K_2))\varphi(S(Y))=x^3(x-1)^3\varphi(S(Y))$.
\item $\varphi(S(\overline{3K_2}\sqcup Y))=\varphi(S(\overline{3K_2}))\varphi(S(Y))=x^3(x+1)^3\varphi(S(Y))$.
\end{enumerate}
\end{cor}

\begin{prop}\label{C}
$\varphi(S(K_3\sqcup Y))=\varphi(S(K_3))\varphi(S(Y))=(x-1)^3\varphi(S(Y))$.
\end{prop}
\begin{proof}
\begin{align*}
\varphi(S(K_3\sqcup Y))
&=\begin{vmatrix}
\begin{array}{ccc|c} 
x&1&1&\\
1&x&1&-J\\
1&1&x&\\ \hline
&-J&&xI-S(Y)\\
\end{array}
\end{vmatrix}
=\begin{vmatrix}
\begin{array}{ccc|c} 
x-1&x-1&x-1&0\cdots 0\\
1&x&1&-1\cdots -1\\
1&1&x&-1\cdots -1\\ \hline
&-J&&xI-S(Y)\\
\end{array}
\end{vmatrix}\\
&=(x-1)\begin{vmatrix}
\begin{array}{ccc|c} 
1&1&1&0\cdots 0\\
1&x&1&-1\cdots -1\\
1&1&x&-1\cdots -1\\ \hline
&-J&&xI-S(Y)\\
\end{array}
\end{vmatrix}=
(x-1)\begin{vmatrix}
\begin{array}{ccc|c} 
1&1&1&0\cdots 0\\
1&x&1&-1\cdots -1\\
1&1&x&-1\cdots -1\\ \hline
&O&&xI-S(Y)\\
\end{array}
\end{vmatrix}\\
&=(x-1)\begin{vmatrix}
\begin{array}{ccc|c} 
1&0&0&0\cdots 0\\
1&x-1&0&-1\cdots -1\\
1&0&x-1&-1\cdots -1\\ \hline
&O&&xI-S(Y)\\
\end{array}
\end{vmatrix}
=(x-1)^3\varphi(S(Y)).
\end{align*}
\end{proof}

\begin{Ex}
Here are the Seidel matrices of the graphs found in the proof of Proposition \ref{D}.
\begin{align*}
S(2K_1)&=\begin{pmatrix}
\begin{array}{cccccc} 
0&1\\
1&0\\
\end{array}
\end{pmatrix}, 
S(K_2)=\begin{pmatrix}
\begin{array}{cccccc} 
0&-1\\
-1&0\\
\end{array}
\end{pmatrix}, 
S(K_2\sqcup K_1)=\begin{pmatrix}
\begin{array}{cccccc} 
0&-1&1\\
-1&0&1\\
1&1&0\\
\end{array}
\end{pmatrix},\\
S(K_2\sqcup 2K_1)&=\begin{pmatrix}
\begin{array}{cccccc} 
0&-1&1&1\\
-1&0&1&1\\
1&1&0&1\\
1&1&1&0\\
\end{array}
\end{pmatrix}, 
S(2K_2)=\begin{pmatrix}
\begin{array}{cccccc} 
0&-1&1&1\\
-1&0&1&1\\
1&1&0&-1\\
1&1&-1&0\\
\end{array}
\end{pmatrix},\\
S(2K_2\sqcup K_1)&=\begin{pmatrix}
\begin{array}{cccccc} 
0&-1&1&1&1\\
-1&0&1&1&1\\
1&1&0&-1&1\\
1&1&-1&0&1\\
1&1&1&1&0\\
\end{array}
\end{pmatrix},
S(2K_2\sqcup 2K_1)=\begin{pmatrix}
\begin{array}{cccccc} 
0&-1&1&1&1&1\\
-1&0&1&1&1&1\\
1&1&0&-1&1&1\\
1&1&-1&0&1&1\\
1&1&1&1&0&1\\
1&1&1&1&1&0\\
\end{array}
\end{pmatrix}.
\end{align*}
\end{Ex}

\begin{prop}\label{D}
For $a,b\in\Z_{\geq 0}, \varphi(S(aK_2\sqcup bK_1))$ is given in the table below.
\begin{center}
\begin{tabular}{|c||c|c|c|c|} \hline
$a\backslash b$&$0\mod 3$&$1\mod 3$\\ \hline\hline
$0\mod 3$&$x^a(x-1)^a(x+1)^b$&$x^{a+1}(x-1)^a(x+1)^{b-1}$\\ \hline
$1\mod 3$&$x^{a-1}(x-1)^a(x+1)^{b+1}$&$x^{a-1}(x-1)^{a+2}(x+1)^{b-1}$\\ \hline
$2\mod 3$&$x^{a-1}(x-1)^{a+1}(x+1)^b$&$x^{a-1}(x-1)^a(x+1)^{b-1}(x^2-x-1)$\\ \hline
\end{tabular}
\begin{tabular}{|c||c|c|c|} \hline
$a\backslash b$&$2\mod 3$\\ \hline\hline
$0\mod 3$&$x^a(x-1)^{a+1}(x+1)^{b-1}$\\ \hline
$1\mod 3$&$x^{a-1}(x-1)^a(x+1)^{b-1}(x^2+1)$\\ \hline
$2\mod 3$&$x^{a-1}(x-1)^a(x+1)^{b-1}(x^2+x-1)$\\ \hline
\end{tabular}
\end{center}
Each cell corresponds to $\varphi(S(aK_2\sqcup bK_1))$ when $a$ satisfies the leftmost congruence condition and $b$ satisfies the uppermost congruence condition.
\end{prop}
\begin{proof}
This follows from Corollary \ref{B}, Proposition \ref{C} and the table below.
\begin{center}
\begin{tabular}{|c||c|c|c|c|} \hline
$a\backslash b$&$0$&$1$&$2$\\ \hline\hline
$0$&$1$&$x$&$(x-1)(x+1)$\\ \hline
$1$&$(x-1)(x+1)$&$(x-1)^3$&$(x-1)(x+1)(x^2+1)$\\ \hline
$2$&$x(x-1)^3$&$x(x-1)^2(x^2-x-1)$&$x(x-1)^2(x+1)(x^2+x-1)$\\ \hline
\end{tabular}
\end{center}
\end{proof}

\section{Main results}

In this section, we state two main theorems. Practically, from Proposition \ref{A}, we get the following relational expression between the adjacency matrix and the Seidel matrix.

\addtocounter{thm}{-2}
\begin{thm}
The following holds.
\begin{enumerate}
\item $\varphi(S(3X),x)=\varphi(A(X),x+1)^3$.
\item $\varphi(S(\overline{3X}),x)=\varphi(A(X),-x+1)^3$.
\end{enumerate}
\end{thm}
\begin{proof}~
\begin{enumerate}
\item By the equation (\ref{E1}) in the proof of Proposition \ref{A},
\begin{align*}
\varphi(S(3X),x)&=\begin{vmatrix}
\begin{array}{c} 
xI-S(X)+J
\end{array}
\end{vmatrix}^3
=\begin{vmatrix}
\begin{array}{c} 
xI-(J-I-2A(X))+J
\end{array}
\end{vmatrix}^3\\
&=\begin{vmatrix}
\begin{array}{c} 
(x+1)I-A(X)
\end{array}
\end{vmatrix}^3
=\varphi(A(X),x+1)^3.
\end{align*}
\item By the equation (\ref{E2}) in the proof of Proposition \ref{A},
\begin{align*}
\varphi(S(\overline{3X}),x)&=\begin{vmatrix}
\begin{array}{c} 
xI+S(X)-J
\end{array}
\end{vmatrix}^3
=\begin{vmatrix}
\begin{array}{c} 
xI+(J-I-2A(X))-J
\end{array}
\end{vmatrix}^3\\
&=\begin{vmatrix}
\begin{array}{c} 
(x-1)I+A(X)
\end{array}
\end{vmatrix}^3
=\varphi(A(X),-x+1)^3.
\end{align*}
\end{enumerate}
\end{proof}

\begin{thm}
Let $r,s,t\in\Z_{\geq 0}$.
\begin{enumerate}
\item If $x^r(x-1)^s(x+1)^t$ is a characteristic polynomial of the Seidel matrix over $\F$ of a graph, then
\begin{itemize}
\item $r\equiv s\equiv t\equiv 0\mod 3$, or
\item $r\equiv 0, s\equiv t\equiv 1\mod 3$, or
\item $r\equiv 1, s\equiv t\equiv 0\mod 3$.
\end{itemize}
\item If $r\leq s+t$ and
\begin{itemize}
\item $r\equiv s\equiv t\equiv 0\mod 3$, or
\item $r\equiv 0, s\equiv t\equiv 1\mod 3$, or
\item $r\equiv 1, s\equiv t\equiv 0\mod 3$,
\end{itemize}
then $x^r(x-1)^s(x+1)^t$ is a characteristic polynomial of the Seidel matrix over $\F$ of a graph.
\end{enumerate}
\end{thm}
\begin{proof}~
\begin{enumerate}
\item Let $c_n$ be the coefficient of $x^{n-2}$ of the characteristic polynomial of the Seidel matrix of order $n$ over $\F$. We show 
\begin{align*}
c_n=\begin{cases}
0&(n\equiv 0,1\mod 3),\\
-1&(n\equiv 2\mod3).
\end{cases}
\end{align*}
By \cite[Theorem 1]{2}, 
\begin{align*}
c_n&=\dfrac{1}{2}\left\{(\mathrm{tr}(S(X)))^2-\mathrm{tr}(S(X)^2)\right\}
=-\dfrac{1}{2}\mathrm{tr}(S(X)^2)
=-\dfrac{1}{2}\sum_{i=1}^{n}\sum_{j=1}^{n}S(X)_{ij}^2\\
&=-\dfrac{1}{2}\sum_{i=1}^{n}(n-1)=-\dfrac{1}{2}n(n-1).
\end{align*}
Hence, the above is satisfied. Next, 
\begin{align*}
(x-1)^s(x+1)^t=\sum_{k=0}^{s}\sum_{l=0}^{t}\dbinom{s}{k}\dbinom{t}{l}(-1)^{s-k}x^{k+l}
\end{align*}
and hence,
\begin{align*}
c_n=\dbinom{s}{s}\dbinom{t}{t-2}-\dbinom{s}{s-1}\dbinom{t}{t-1}+\dbinom{s}{s-2}\dbinom{t}{t}=\dfrac{1}{2}s(s-1)-st+\dfrac{1}{2}t(t-1).
\end{align*}
The sum of the eigenvalues of a matrix is equal to its trace. Thus, we  may consider $s\equiv t\mod 3$ and examine all cases:
\begin{itemize}
\item $r\equiv s\equiv t\equiv 0\Longrightarrow n\equiv 0, c_n\equiv 0$.
\item $r\equiv 0, s\equiv t\equiv 1\Longrightarrow n\equiv 2, c_n\equiv -1$.
\item $r\equiv 0, s\equiv t\equiv 2\Longrightarrow n\equiv 1, c_n\equiv 1$.
\item $r\equiv 1, s\equiv t\equiv 0\Longrightarrow n\equiv 1, c_n\equiv 0$.
\item $r\equiv s\equiv t\equiv 1\Longrightarrow n\equiv 0, c_n\equiv -1$.
\item $r\equiv 1, s\equiv t\equiv 2\Longrightarrow n\equiv 2, c_n\equiv 1$.
\item $r\equiv 2, s\equiv t\equiv 0\Longrightarrow n\equiv 2, c_n\equiv 0$.
\item $r\equiv 2, s\equiv t\equiv 1\Longrightarrow n\equiv 1, c_n\equiv -1$.
\item $r\equiv s\equiv t\equiv 2\Longrightarrow n\equiv 0, c_n\equiv 1$.
\end{itemize}
Thus, the proof is completed.
\item Let $a,b\in\{0,1,2\},c,d,e,f\in\Z_{\geq 0}$, where $[a=1\Rightarrow b\in\{0,1\}]$ and $[a=2\Rightarrow b=0]$. We consider a graph of the form
\begin{align*}
\mathfrak{X}:=aK_2\sqcup bK_1\sqcup cK_3\sqcup d\overline{K_3}\sqcup e(3K_2)\sqcup f(\overline{3K_2}).
\end{align*}
Then, by Corollary \ref{B}, Proposition \ref{C} and Proposition \ref{D},
\begin{align}\label{S(X)}
\begin{split}
\varphi(S(\mathfrak{X}))=\begin{cases}
x^{3(e+f)}(x-1)^{3(c+e)}(x+1)^{3(d+f)}&(a=b=0),\\
x^{3(e+f)}(x-1)^{3(c+e+1)}(x+1)^{3(d+f)}&(a=b=1),\\
x^{3(e+f)}(x-1)^{3(c+e)+1}(x+1)^{3(d+f)+1}&(a=1,b=0; a=0,b=1),\\
x^{3(e+f)+1}(x-1)^{3(c+e+1)}(x+1)^{3(d+f)}&(a=2,b=0; a=0,b=2).
\end{cases}
\end{split}
\end{align}
Also, by $\varphi(S(4K_1))=x(x+1)^3$, 
\begin{align*}
\varphi(S(4K_1\sqcup f(\overline{3K_2})))=x^{3f+1}(x+1)^{3(f+1)}
\end{align*}
and the proof is completed.
\end{enumerate}
\end{proof}

\section{Further consideration}

In this section, 

Even if $r>s+t$, the following cases are possible.

\begin{prop}\label{propn0}
For $k\in\Z_{\geq 0}$, the following holds.
\begin{enumerate}
\item If $s\equiv t\equiv 0\mod 3$ and 
\begin{itemize}
\item $r=27\dbinom{k+1}{2}$ and $[s\neq 0$ or $t\neq 0]$, or
\item $r=27\dbinom{k+1}{2}+1$ and $[s, t\neq 0$ or $s\neq 0,3]$, or
\item $\Big[r=27\dbinom{k}{2}$ or $r=3\dbinom{3k}{2}\Big]$ and $[s\geq 9k$ or $t\geq 9k]$, or
\item $\Big[r=27\dbinom{k}{2}+1$ or $r=3\dbinom{3k}{2}+1\Big]$ and $[s\geq 3(3k+1)$ or $[s\neq 0$ and $t\geq 9k]]$,
\end{itemize}
then $x^r(x-1)^s(x+1)^t$ is a characteristic polynomial of the Seidel matrix over $\F$ of a graph.
\item If $s\equiv t\equiv 1\mod 3$ and 
\begin{itemize}
\item $r=27\dbinom{k+1}{2}$ and $[s\neq 1$ or $t\neq 1]$, or
\item $r=27\dbinom{k}{2}$ and $[s\geq 9k+1$ or $t\geq 9k+1]$, or
\item $r=3\dbinom{3k}{2}$ and $[s\geq 9k+1$ or $t\geq 9k+1]$,
\end{itemize}
then $x^r(x-1)^s(x+1)^t$ is a characteristic polynomial of the Seidel matrix over $\F$ of a graph.
\end{enumerate}
\end{prop}
\begin{proof}~
By \cite[Lemma 8.2.5]{1},
\begin{align*}
\varphi(A(L(X)),x)=(x+2)^{e-n}\varphi(A(X), x-k+2)
\end{align*}
for a $k$-regular graph $X$ consisting of $n$ vertices and $e$ edges, where $L(X)$ is the line graph of $X$. Therefore, from Theorem \ref{th1},
\begin{align*}
\varphi(S(3L(X)),x)=x^{3(e-n)}\varphi(A(X), x-k)^3.
\end{align*}
Since $\varphi(A(K_n),x)=(x-n+1)(x+1)^{n-1}$,
\begin{align*}
\varphi(S(3L(K_n)),x)=x^{\frac{3}{2}n(n-3)}(x-2n+2)^3(x-n+2)^{3(n-1)}.
\end{align*}
By Theorem \ref{th1}, $\varphi(S(\overline{3L(K_n)}),x)=\varphi(S(3L(K_n)),-x)$.
In the equation (\ref{S(X)}), we look at each of the cases and the proof is completed: If $n\equiv 0\mod 3$, then
\begin{itemize}
\item $a=b=0, e=f=0\\
\Longrightarrow
\varphi(S(3L(K_n)\sqcup\mathfrak{X}),x)=x^{\frac{3}{2}n(n-3)}(x-1)^{3(n+c)}(x+1)^{3d}, \\\hspace*{0.75cm}\varphi(S(\overline{3L(K_n)}\sqcup\mathfrak{X}),x)=x^{\frac{3}{2}n(n-3)}(x-1)^{3c}(x+1)^{3(n+d)}$.
\item $a=1, b=0, e=f=0\\
\Longrightarrow
\varphi(S(3L(K_n)\sqcup\mathfrak{X}),x)=x^{\frac{3}{2}n(n-3)}(x-1)^{3(n+c)+1}(x+1)^{3d+1}, \\\hspace*{0.75cm}\varphi(S(\overline{3L(K_n)}\sqcup\mathfrak{X}),x)=x^{\frac{3}{2}n(n-3)}(x-1)^{3c+1}(x+1)^{3(n+d)+1}$.
\item $a=2, b=0, e=f=0\\
\Longrightarrow
\varphi(S(3L(K_n)\sqcup\mathfrak{X}),x)=x^{\frac{3}{2}n(n-3)+1}(x-1)^{3(n+c+1)}(x+1)^{3d}, \\\hspace*{0.75cm}\varphi(S(\overline{3L(K_n)}\sqcup\mathfrak{X}),x)=x^{\frac{3}{2}n(n-3)+1}(x-1)^{3(c+1)}(x+1)^{3(n+d)}$.
\end{itemize}
If $n\equiv 1\mod 3$, then
\begin{itemize}
\item $a=b=0, e=f=0\\
\Longrightarrow
\varphi(S(3L(K_n)\sqcup\mathfrak{X}),x)=x^{\frac{3}{2}(n-1)(n-2)}(x-1)^{3c}(x+1)^{3(n+d-1)}, \\\hspace*{0.75cm}\varphi(S(\overline{3L(K_n)}\sqcup\mathfrak{X}),x)=x^{\frac{3}{2}(n-1)(n-2)}(x-1)^{3(n+c-1)}(x+1)^{3d}$.
\item $a=1, b=0, e=f=0\\
\Longrightarrow
\varphi(S(3L(K_n)\sqcup\mathfrak{X}),x)=x^{\frac{3}{2}(n-1)(n-2)}(x-1)^{3c+1}(x+1)^{3(n+d-1)+1}, \\\hspace*{0.75cm}\varphi(S(\overline{3L(K_n)}\sqcup\mathfrak{X}),x)=x^{\frac{3}{2}(n-1)(n-2)}(x-1)^{3(n+c-1)+1}(x+1)^{3d+1}$.
\item $a=2, b=0, e=f=0\\
\Longrightarrow
\varphi(S(3L(K_n)\sqcup\mathfrak{X}),x)=x^{\frac{3}{2}(n-1)(n-2)+1}(x-1)^{3(c+1)}(x+1)^{3(n+d-1)}, \\\hspace*{0.75cm}\varphi(S(\overline{3L(K_n)}\sqcup\mathfrak{X}),x)=x^{\frac{3}{2}(n-1)(n-2)+1}(x-1)^{3(n+c)}(x+1)^{3d}$.
\end{itemize}
If $n\equiv 2\mod 3$, then
\begin{itemize}
\item $a=b=0, e=f=0\\
\Longrightarrow
\varphi(S(3L(K_n)\sqcup\mathfrak{X}),x)=x^{\frac{3}{2}(n+1)(n-2)}(x-1)^{3c}(x+1)^{3(d+1)}, \\\hspace*{0.75cm}\varphi(S(\overline{3L(K_n)}\sqcup\mathfrak{X}),x)=x^{\frac{3}{2}(n+1)(n-2)}(x-1)^{3(c+1)}(x+1)^{3d}$.
\item $a=1, b=0, e=f=0\\
\Longrightarrow
\varphi(S(3L(K_n)\sqcup\mathfrak{X}),x)=x^{\frac{3}{2}(n+1)(n-2)}(x-1)^{3c+1}(x+1)^{3(d+1)+1}, \\\hspace*{0.75cm}\varphi(S(\overline{3L(K_n)}\sqcup\mathfrak{X}),x)=x^{\frac{3}{2}(n+1)(n-2)}(x-1)^{3(c+1)+1}(x+1)^{3d+1}$.
\item $a=2, b=0, e=f=0\\
\Longrightarrow
\varphi(S(3L(K_n)\sqcup\mathfrak{X}),x)=x^{\frac{3}{2}(n+1)(n-2)+1}(x-1)^{3(c+1)}(x+1)^{3(d+1)}, \\\hspace*{0.75cm}\varphi(S(\overline{3L(K_n)}\sqcup\mathfrak{X}),x)=x^{\frac{3}{2}(n+1)(n-2)+1}(x-1)^{3(c+2)}(x+1)^{3d}$.
\end{itemize}
\end{proof}

\section*{Acknowledgements}
I am grateful to Professor A. Higashitani for helpful discussions and carefully proofreading the manuscript.

\end{document}